\documentclass[a4paper,12pt,reqno]{amsart}
\usepackage{mymacros}
\usepackage[
    inline,
    final
    ]{ showlabels }

\newcommand{\baselocus}{\operatorname{\mathsf{Bs}}}
\newcommand{\parabaselocus}{\operatorname{\mathsf{PBs}}}

\newcommand{\exceptionalset}{\mathsf{Exc}}

\title{Semiorthogonal indecomposability of minimal irregular surfaces}

\date{\today}

\author{Shinnosuke Okawa}
\address{
Department of Mathematics,
Graduate School of Science,
Osaka University,
Machikaneyama 1--1,
Toyonaka,
Osaka,
560--0043,
Japan.
}
\email{okawa@math.sci.osaka-u.ac.jp}

\begin{document}

\begin{abstract}
    We prove a relative version of the fact that semiorthogonal decompositions of the bounded derived category of coherent sheaves are strongly constrained by the base locus of the canonical linear system. As an application we prove that the derived category of minimal surfaces
    \(
       X
    \)
    with
    \(
        H ^{ 1 } (
            X,
            \cO
            _{
                X
            }
        )
        \neq
        0
    \)
    are semiorthogonally indecomposable.
\end{abstract}

\maketitle

\tableofcontents
%%%%%%%%%%%

% \setcounter{section}{-1}

% %
% %------------------------------------------------------------------------
% %
% \section{Notation and conventions}

% Throughout the paper we work over an algebraically closed field \( \bfk \) of an arbitrary characteristic.

%
%------------------------------------------------------------------------
%
\section{Introduction}

It is impossible to stress enough the significance of the relative point of view for the minimal model theory.
% It is quite important to consider the birational geometry of
% \(
%    X
% \)
% relative to
% \(
%    f
% \), which means to work in the category of schemes and rational maps over
% \(
%    Y
% \).
Even if one is only interested in the birational geometry of a variety
\(
   X
\), in many places it is useful to choose an appropriate morphism
\(
   f
   \colon
   X
   \to
   Y
\)
and consider the geometry of
\(
   X
\)
relative to
\(
   f
\).

This paper is devoted to bringing this insight into the study of semiorthogonal decomposition (SOD) of the bounded derived category of coherent sheaves and demonstrating its usefulness by showing the semiorthogonal indecomposability of minimal surfaces with positive irregularity.

% Semiorthogonal decomposition (SOD) is a basic method for understanding triangulated categories. It decomposes a triangulated category into several triangulated subcategories, or put differently, present it as several smaller triangulated categories glued together.
% It is hence important to clarify which triangulated categories are atomic, in the sense that they admit no semiorthogonal decompositions.

A guiding principal for the study of SODs of the bounded derived category
\(
   \derived
   (
    X
   )
   =
   \derived
   ^{
    \bounded
   }
   \coh
   X
\)
of coherent sheaves on a smooth projective variety
\(
   X
\)
is the \emph{DK hypothesis}. It asserts that a step of a minimal model program
\(
   X
   \dasharrow
   Y
\)
(or more generally a \(K\)-inequality
\(
   K
   _{
    X
   }
   >
   K
   _{
    Y
   }
\))
induces an SOD of
\(
   \derived
   (
    X
   )
\)
by
``\(
   \derived
   (
    Y
   )
\)''
and some other components; since
\(
   Y
\)
could be singular even if
\(
   X
\)
is nonsingular, part of the hypothesis is to define the category
``\(
   \derived
   (
    Y
   )
\)''
suitably (see, say,~\cite[]{MR3838122} and references therein). Note that the hypothesis implies that if
\(
   \derived
   (
    X
   )
\)
is indecomposable, then
\(
   X
\)
should be minimal.

On the other hand, the other implication

\noindent
\begin{align}\label{equation:SOD implies MMP}
    \mathrm{SOD}
    \Rightarrow
    \mathrm{MMP}
\end{align}
does not always hold in dimensions at least
\(
   2
\):

\begin{example}\label{example:exceptional object}
    Suppose that
    \(
       H
       ^{
        i
       }
       (
        X,
        \cO
        _{
            X
        }
       )
       =
       0
    \)
    for
    \(
       i
       >
       0
    \).
    Then any invertible sheaf on
    \(
       X
    \)
    is an exceptional object and hence decomposes
    \(
       \derived
       (
        X
       )
    \).
    There are several examples of minimal surfaces of Kodaira dimensions
    \(
       1
    \)
    or
    \(
       2
    \)
    of this sort. They are rich sources of (quasi-)phantom categories (see, say,~\cite{2015arXiv150800682K,MR3728631,MR3090263} and references therein).
\end{example}

Nevertheless, known results seem to suggest the following bold hypothesis.

\begin{hypothesis}\label{hypothesis:bold hypothesis}
    For almost all smooth projective variety
    \(
       X
    \)
    which is minimal, the derived category
    \(
       \derived
       (
        X
       )
    \)
    admits no nontrivial semiorthogonal decomposition.
\end{hypothesis}

Making the phrase `almost all' in~\cref{hypothesis:bold hypothesis} as precise as possible is part of the problem.
Note that it is a special case of the following even bolder hypothesis. One can tell from experience, however, that~\cref{hypothesis:bold hypothesis} is the essential part (see~\cite[Proposition~5.7]{2020arXiv200607643P}).

\begin{hypothesis}\label{hypothesis:bolder hypothesis}
    For almost all smooth projective variety
    \(
       X
    \)
    which is \emph{not necessarily} minimal, any SOD of the derived category
    \(
       \derived
       (
        X
       )
    \)
    comes from MMPs starting with
    \(
       X
    \)
    through the DK hypothesis up to autoequivalences of
    \(
       \derived
       (
        X
       )
    \).
    Namely, for almost all
    \(
       X
    \)
    the implication~\eqref{equation:SOD implies MMP} holds (up to autoequivalences).
\end{hypothesis}

\cref{hypothesis:bolder hypothesis} is completely settled for curves in~\cite[]{MR2838062}. See~\cite[]{2015arXiv150800682K} and the rest of this introduction for what is known in dimensions
\(
   \ge
   2
\).

In this paper we add a new technique for investigating SODs of
\(
   \derived
   (
    X
   )
\),
which is based on the \emph{relative point of view}.
Let
\(
   X
\)
be a nonsingular variety and
\(
   f
   \colon
   X
   \to
   Y
\)
be a projective morphism to another variety (or more generally an algebraic space). As we mentioned in the beginning of this introduction, it is quite important to consider the birational geometry of
\(
   X
\)
over
\(
   Y
\).

We claim that the notion of \emph{\(f\)-linear semiorthogonal decompositions} (\cref{definition:semiorthogonal decomposition}) is the categorical counterpart of birational geometry over
\(
   Y
\).
The correspondence is summarized in~\cref{figure:dictionary} below.

\begin{figure}[H]
    \begin{tabular}{c c c}
        \toprule
        &
        geometry
        &
        category
        \\
        % \cmidrule(r){1-1}
        \cmidrule(r){2-2}
        \cmidrule(r){3-3}
        absolute
        &
        MMP of
        \(
            X
        \)
        &
        SOD of
        \(
            \derived
            (
            X
            )
        \)
        \\
        relative
        &
        MMP of
        \(
            X
        \)
        over
        \(
            Y
        \)
        &
        \(
            f
        \)-linear SOD of
        \(
            \derived
            (
            X
            )
        \)
        \\
        \bottomrule
    \end{tabular}
    \caption{Dictionary}\label{figure:dictionary}
\end{figure}

This, for example, is supported by the following result due to Pirozhkov, which plays an important role in this paper.

\begin{theorem}[{\(=\)\cite[Theorem~4.5]{2020arXiv201112743P}, simplified for brevity}]\label{theorem:Pirozhkov}
    Let
    \(
       f
       \colon
       X
       \to
       Y
    \)
    be a morphism of projective varieties such that
    \(
       X
    \)
    is nonsingular. Suppose that
    \(
       Y
    \)
    admits a finite morphism to an abelian variety. Then any SOD of
    \(
       \derived
       (
        X
       )
    \)
    is
    \(
       f
    \)-linear.
\end{theorem}

\begin{remark}
    \cref{theorem:Pirozhkov} is comparable to the fact that, in the same situation, any MMP starting with
    \(
        X
    \)
    is over
    \(
        Y
    \)
    since there is no rational curve on an abelian variety.
\end{remark}

In~\cite[Theorem~1.2]{2015arXiv150800682K} Kawatani and the author proved that semiorthogonal decompositions of a smooth projective variety
\(
   X
\)
is strongly constrained by the base locus of the canonical linear system. A main result of this paper is that the straightforward generalization to the relative settings holds as follows.

\begin{theorem}[{\(=\)\cref{theorem:dichotomy}}]\label{theorem:dichotomy-Introduction}
    Let
    \(
        f \colon X \to Y
    \)
    be a proper morphism of finite dimensional noetherian algebraic spaces such that \( X \) is a regular connected scheme. Let
    
    \noindent
    \begin{align}
        \derived ( X ) = \langle \cA, \cB \rangle
    \end{align}
    be an \( f \)-linear semiorthogonal decomposition. Then one of the followings holds, where
    \(
        \baselocus _{ f } ( \omega _{ X } )       
    \)
    is the relative canonical base locus of
    \(
       X
    \)
    (see~\cref{definition:relative base locus}).
    
    \noindent
    \begin{enumerate}[(A)]
        \item\label{item:points belong to A_Introduction}
        For any closed point \( x \nin \baselocus _{ f } ( \omega _{ X } ) \) the sheaf \( \bfk ( x )\) is contained in \( \cA \).

        \item\label{item:points belong to B_Introduction}
        For any closed point \( x \nin \baselocus _{ f } ( \omega _{ X } ) \) the sheaf \( \bfk ( x )\) is contained in \( \cB \).
    \end{enumerate}
    Moreover, when~\cref{item:points belong to A_Introduction} (resp.~\cref{item:points belong to B_Introduction}) holds, then the support of any object of \( \cB \) (resp. \( \cA \)) is contained in the closed subset
    \(
        \baselocus _{ f } ( \omega _{ X } )
    \).
\end{theorem}

The case where
\(
   Y
   =
   \Spec \bfk
\)
of~\cref{theorem:dichotomy-Introduction} is exactly the same as~\cite[Theorem~1.2]{2015arXiv150800682K},
where the relative canonical base locus is nothing but the usual canonical base locus.

Note that~\cref{theorem:dichotomy-Introduction} is only about those semiorthogonal decompositions of
\(
   X
\)
which are
\(
   f
\)-linear. However, if we are in the situation of~\cref{theorem:Pirozhkov}, then any semiorthogonal decomposition of
\(
   X
\)
is automatically
\(
   f
\)-linear. Since the relative canonical base locus is smaller than the usual canonical base locus by~\cref{lemma:transitivity}, and in some cases they do differ, we obtain stronger constrains from~\cref{theorem:dichotomy-Introduction} than~\cite[Theorem~1.2]{2015arXiv150800682K}. See~\cref{example:PBs as relative canonical base locus} for the example which motivated this work.

When the variety of interest is of positive irregularity,
then we obtain for free a nontrivial morphism to an abelian variety; i.e., the Albanese morphism. With this in mind, as an application of~\cref{theorem:dichotomy-Introduction}, we prove the following

\begin{theorem}[{\(=\)\cref{theorem:irregular surface is SI}}]\label{theorem:main}
    Let
    \(
       X
    \)
    be a minimal surface with
    \(
        H
        ^{
            1
        }
        \left(
            X,
            \cO
            _{
                X
            }
        \right)
        \neq
        0
    \).
    Then
    \(
        \derived
        (
            X
        )
    \)
    admits no non-trivial semiorthogonal decomposition.
\end{theorem}

The proof goes as follows (see~\cref{section:Semiorthogonal indecomposability of irregular minimal surfaces} for details).
The assumptions of~\cref{theorem:main} imply that
\(
   X
\)
admits a morphism to a variety of positive dimension such that any semiorthogonal decomposition of
\(
   \derived
   (
    X
   )
\)
is linear with respect to the morphism. Hence, as we explained just before~\cref{theorem:main}, it amounts to understanding the relative canonical base loci of such morphisms.
Thankfully, for surfaces Kazuhiro Konno established structure theorems for the relative canonical base loci which are strong enough for our purpose. Thus we can conclude the proof.

The strategy itself, however, applies to varieties of higher dimensions of positive irregularity as well. Though we leave it for future, it would be interesting to use the same strategy to obtain semiorthogonal indecomposability in dimensions
\(
   \ge
   3
\).

In view of~\cref{theorem:main} and the other results for surfaces established in~\cite[]{2015arXiv150800682K}, it seems fair to make the following conjecture. This is~\cref{hypothesis:bold hypothesis} made precise in dimension
\(
   2
\).

\begin{conjecture}\label{conjecture:main conjecture}
    A minimal surface
    \(
       X
    \)
    is not semiorthogonally indecomposable if and only if
    \(
       \cO
       _{
        X
       }
    \)
    is acyclic; namely,
    \(
       H
       ^{
        i
       }
       \left(
        X,
        \cO
        _{
            X
        }
       \right)
       =
       0
    \)
    for
    \(
       i
       =
       1, 2
    \).
\end{conjecture}

In view of~\cref{example:exceptional object}, it seems natural to generalize~\cref{conjecture:main conjecture} to the following hypothesis, which partly refines~\cref{hypothesis:bolder hypothesis} in dimension
\(
   2
\).

\begin{hypothesis}\label{hypothesis:bolder hypothesis in dimension 2}
    Let
    \(
       X
    \)
    be a surface of non-negative Kodaira dimension. Then any semiorthogonal decomposition of
    \(
       \derived
       (
        X
       )
    \), if any, is induced by exceptional collections.
\end{hypothesis}

\cref{corollary:remaining cases_Introduction} below summarizes the remaining cases of~\cref{conjecture:main conjecture}.

\begin{corollary}[{\(=\)\cref{corollary:remaining cases}}]\label{corollary:remaining cases_Introduction}
    \cref{conjecture:main conjecture} is true except possibly for those
    \(
       X
    \)
    which satisfy all of the following conditions.

    \noindent
    \begin{itemize}
        \item \( X \) is a minimal surface of general type.
        \item \( H ^{ 2 } ( X, \cO _{ X } ) \neq 0 \) (\( \iff H ^{ 0 } ( X, \omega _{ X } ) \neq 0 \)).
        \item \( H ^{ 1 } ( X, \cO _{ X } ) = 0 \).
        \item There exists a 1-dimensional component of
        \(
           \baselocus
           \left(
            \omega _{ X }
           \right)
        \)
        which can not be contracted to a point in the category of algebraic spaces.
    \end{itemize}
\end{corollary}

In order to tackle
\(
   X
\)
as in~\cref{corollary:remaining cases_Introduction}, we need a better understanding of the (absolute) canonical base locus
\(
   \baselocus
   (
    \omega
    _{
        X
    }
   )
\).
Unfortunately this does not seem to have been worked out in the literature, as it is mentioned as a very interesting question in the end of~\cite[Introduction]{MR2654305}.

%
%------------------------------------------------------------------------
%

\subsection*{Relations to other works}

This paper is a continuation of the series of works~\cite[]{MR2838062,2015arXiv150800682K}, which aims at~\cref{hypothesis:bolder hypothesis}. This paper came as a result of an attempt to interpret from a natural point of view the work~\cite[Theorem~1.5]{2021arXiv211006795C}, which identifies the paracanonical base locus as the relative base locus of the canonical line bundle with respect to the Albanese morphism, where the importance of the former in the context of semiorthogonal decompositions was first discovered in~\cite[Theorem~1.4]{2021arXiv210709564L}. See~\cref{example:PBs as relative canonical base locus} for details.

\cref{theorem:irregular surface is SI} is proved for some examples of surfaces of general type in~\cite[Theorem~4.16]{2021arXiv210709564L}, and Caucci~informed the author that they also obtained~\cref{theorem:irregular surface is SI} for surfaces of general type of maximal Albanese dimension (private communication).

In~\cref{section:Alternative proof via relative Hochschild homology} we give an alternative proof of~\cref{theorem:dichotomy} using the Hochschild homology, which is a relative version of what is discussed in~\cite[Section~5]{2020arXiv200607643P}. As is the case for the results of~\cite[Section~5]{2020arXiv200607643P}, actually we get slightly finer information than~\cref{theorem:dichotomy}. Though we do not pursuit it in this paper, it is interesting to ask if we can obtain results about~\cref{hypothesis:bolder hypothesis} analogous to~\cite[Proposition~5.7]{2020arXiv200607643P} from it.

During the preparation of this paper, an independent work~\cite[]{2023arXiv230106818L} appeared on the arXiv. Their main result~\cite[Theorem~2.3]{2023arXiv230106818L} is more general than~\cref{theorem:dichotomy-Introduction}, in that they do not assume that
\(
   X
\)
is nonsingular. See~\cref{remark:comparison} for detailed comparison.

%
%------------------------------------------------------------------------
%
\subsection*{Acknowledgements}

The author is indebted to Ana Cristina L\'opez Mart\'\i n and Fernando Sancho de Salas for their various comments on this paper, especially as to dualizing complexes, and valuable exchanges, Alexander Kuznetsov for various useful comments, and Federico Caucci for many comments and corrections. He also thanks Kotaro Kawatani, Kazuhiro Konno, and Xun Lin for useful discussions.

The author was partially supported by JSPS Grants-in-Aid for Scientific Research
(18H01120, %Prof. Goto
19KK0348, %Kokusai
20H01797, %Kiban B
20H01794, %Alexey
21H04994%Prof. Takahashi
).

%
%------------------------------------------------------------------------
%
\section{Basics of relative semiorthogonal decomposition}

The aim of this section is to generalize the rudiments of semiorthogonal decompositions to the relative setting as much as possible.
% These results will be used later to investigate their relationship to the relative canonical base loci.

\begin{definition}\label{definition:D(X)}
    Let \( X \) be a regular noetherian scheme of finite dimension. We let
    \(
    \derived ( X )
    \)
    denote the bounded derived category of coherent sheaves on \( X \).
\end{definition}

By the theorem of Serre, the natural embedding
\(
\perf ( X ) \hookrightarrow \derived ( X )
\)
is full for \( X \) as in~\cref{definition:D(X)}.

We use the Grothendieck duality in an essential way:

\begin{theorem}[{\(=\)\cite[Chapter~VII~Corollary~3.4~(c)]{MR0222093}}]\label{theorem:Grothendieck duality}
    Let

    \noindent
    \begin{align}
        f \colon X \to Y
    \end{align}
    be a proper morphism of finite dimensional noetherian schemes.
    Then there is a functor

    \noindent
    \begin{align}
        f ^{ ! }
        \colon
        \derived ^{ + } _{ \coh } ( Y )
        \to
        \derived ^{ + } _{ \coh } ( X )
    \end{align}
    which is right adjoint to the derived pushforward functor

    \noindent
    \begin{align}
        f _{ \ast } \colon \derived ^{ + } _{ \Qcoh } ( X )
        \to
        \derived ^{ + } _{ \Qcoh } ( Y )
    \end{align}
    in the sense that there exists a bi-functorial isomorphism in
    \(
        \derived ^{ + } _{ \Qcoh } ( Y )
    \)
    as follows for any
    \(
        \cF \in \derived ^{ - } _{ \Qcoh } ( X )
    \)
    and
    \(
        \cG
        \in
        \derived ^{ + } _{ \coh } ( Y )
    \).

    \noindent
    \begin{align}
        f _{ \ast } \cRHom _{ X } ( \cF, f ^{ ! } \cG )
        \to
        \cRHom _{ Y } ( f _{ \ast } \cF, \cG )
    \end{align}
\end{theorem}

The following special case of~\cref{theorem:Grothendieck duality} is particularly relevant to us (see~\cref{definition:dualizing complex,remark:relative canonical base locus} for the notion of dualizing complexes).

\begin{corollary}\label{corollary:Grothendieck duality}
    In the situation of~\cref{theorem:Grothendieck duality},
    there is a natural isomorphism as follows for any
    \(
        F, G \in \perf ( X )
    \), where
    \(
        \omega ^{ \bullet } _{ X },
        \omega ^{ \bullet } _{ Y }
    \)
    are dualizing complexes of \( X \) and \( Y \), respectively, satisfying
    \(
        f
        ^{
            !
        }
        \omega
        ^{
            \bullet
        }
        _{
            Y
        }
        \simeq
        \omega
        ^{
            \bullet
        }
        _{
            X
        }
    \).

    \noindent
    \begin{align}
        f _{ \ast }\cRHom _{ X } ( G, F \otimes \omega ^{ \bullet } _{ X } )
        \simeq
        \cRHom _{ Y } ( f _{ \ast } \cRHom _{ X } ( F, G ), \omega ^{ \bullet } _{ Y } )
    \end{align}
\end{corollary}

\begin{proof}
    Apply~\cref{theorem:Grothendieck duality} to
    \(
        \cF = F ^{ \vee } \otimes G
    \)
    and
    \(
        \cG = \omega ^{ \bullet } _{ Y }
    \).
\end{proof}

\begin{definition}\label{definition:semiorthogonal decomposition}
    Let \( f \colon X \to Y \) be a projective morphism of noetherian algebraic spaces such that
    \(
       X
    \)
    is regular and of finite dimension. An
    \emph{\( f \)-linear semiorthogonal decomposition}
    (alternatively
    \emph{
        semiorthogonal decomposition of \( \derived ( X ) \) relative to \( f \)
    }
    or
    \emph{
        semiorthogonal decomposition of \( \derived ( X ) \) over \( Y \)
    })
    of
    \(
        \derived ( X )
    \)
    (of length \( 2 \)) is a pair of (full) triangulated subcategories
    \(
        \cA, \cB
        \subseteq
        \derived ( X )
    \)
    such that
    
    \begin{itemize}
        \item \( \derived ( X ) \) coincides with the smallest triangulated subcategory containing both \( \cA \) and \( \cB \), and
        \item The pair is semiorthogonal over \( Y \), in the sense that
              \begin{align}
                  f _{ \ast } \cRHom _{ X } ( b, a ) = 0
              \end{align}
              for any \( a \in \cA \) and \( b \in \cB \).
    \end{itemize}
    We use the shorthand notation
    \begin{align}\label{equation:D(X)=<A,B>}
        \derived ( X )
        =
        \langle
            \cA,
            \cB
        \rangle
        _{
            f
        }
    \end{align}
    to indicate the situation.

    An
    \(
       f
    \)-linear semiorthogonal decomposition for the structure morphism
    \(
       f
       \colon
       X
       \to
       \Spec \bZ
    \)
    is simply called
    \emph{
        semiorthogonal decomposition
    }
    and denoted by
    
    \noindent
    \begin{align}
        \derived
        (
            X
        )
        =
        \langle
            \cA,
            \cB
        \rangle.
    \end{align}
\end{definition}

\begin{remark}\label{remark:transitivity}
    In the situation of~\cref{definition:semiorthogonal decomposition}, let
    \(
       g
       \colon
       Y
       \to
       Z
    \)
    be a projective morphism of noetherian algebraic spaces and put
    \(
       h
       =
       g
       \circ
       f
    \).
    Then an
    \(
       f
    \)-linear semiorthogonal decomposition of
    \(
       X
    \)
    is automatically an
    \(
       h
    \)-linear semiorthogonal decomposition.
    In particular, any
    \(
       f
    \)-linear semiorthogonal decomposition is a semiorthogonal decomposition.
\end{remark}

\begin{definition}\label{definition:orthogonal decomposition}
    In the situation of~\cref{definition:semiorthogonal decomposition}, an
    \(
       f
    \)-linear semiorthogonal decomposition~\eqref{equation:D(X)=<A,B>} is said to be an~\emph{orthogonal decomposition} if it also satisfies
    
    \noindent
    \begin{align}
        \derived
        (
            X
        )    
        =
        \langle
            \cB,
            \cA    
        \rangle
        _{
            f
        }.
    \end{align}
    We indicate the situation by
    
    \noindent
    \begin{align}
        \derived
        (
            X
        )    
        =
        \cA
        \perp
        _{
            f
        }
        \cB.
    \end{align}
\end{definition}

\begin{definition}\label{definition:semiorthogonal indecomposability}
    % We say that \( X \) is \emph{semiorthogonally indecomposable (SI)} if
    % \(
    % \derived ( X )
    % \)
    % admits no nontrivial semiorthogonal decomposition in the sense that any semiorthogonal decomposition is an orthogonal decomposition.
    In the situation~\cref{definition:semiorthogonal decomposition} we say that \emph{\( X \) is semiorthogonally indecomposable over \( Y \) (SI over \( Y \))} or \emph{\( f \) is semiorthogonally indecomposable (\(f\) is SI)} if any \( f \)-linear semiorthogonal decomposition of \( X \) is an orthogonal decomposition.
    We say \( X \) is semiorthogonally indecomposable (SI) if it is so over
    \(
        \Spec \bZ
    \).
\end{definition}

\begin{remark}
    If \( X \) is connected, then \( X \) being SI implies that any semiorthogonal decomposition of
    \(
    \derived ( X )
    \)
    is trivial; i.e., there is no semiorthogonal decomposition of \( \derived ( X ) \) other than
    \(
    \langle 0, \derived ( X ) \rangle
    \)
    and
    \(
    \langle \derived ( X ), 0 \rangle
    \).
\end{remark}

As we pointed out in~\cref{remark:transitivity},
\(
   f
\)-linear semiorthogonal decompositions are semiorthogonal decompositions satisfying an additional condition.~\cref{lemma:characterization of f-linearity} below, which is a slight generalization of~\cite[Lemma~2.7]{MR2801403}, characterizes the condition in a slightly different way. 
This guarantees that~\cref{definition:semiorthogonal decomposition} is equivalent to the original definition of
\(
   f
\)-linearity given in~\cite{MR2238172,MR2801403}.

\begin{lemma}\label{lemma:characterization of f-linearity}
    Suppose that
    \(
       Y
    \)
    is a quasi-separated and quasi-compact scheme. Then a semiorthogonal decomposition
    \(
    \derived ( X ) = \langle \cA, \cB \rangle
    \)
    is \( f \)-linear if and only if \( \cA \) (resp. \( \cB \)) is stable under the action of
    \(
    \perf ( Y )
    \), in the sense that
    \(
    a \otimes f ^{ \ast } y \in \cA
    \)
    for all
    \(
    a \in \cA
    \)
    and
    \(
    y \in \perf ( Y )
    \)
    (resp.
    \(
    b \otimes f ^{ \ast } y \in \cB
    \)
    for all
    \(
    b \in \cB
    \)
    and
    \(
    y \in \perf ( Y )
    \)).
\end{lemma}

\begin{proof}
    If the semiorthogonal decomposition is \( f \)-linear, then
    \begin{align}
        f _{ \ast } \cRHom _{ X } ( b, a ) \otimes y = 0
    \end{align}
    for all
    \(
    a \in \cA, b \in \cB, y \in \perf ( Y )
    \).
    Then \( \bR \Gamma ( Y, ? ) \) of the left hand side is isomorphic to
    \(
    \RHom _{ X } ( b, a \otimes f ^{ \ast } y )
    \). Thus we have confirmed that
    \(
    a \otimes f ^{ \ast } y
    \in
    \cB ^{ \perp }
    =
    \cA
    \).
    Conversely the vanishing of
    \begin{align}
        \RHom _{ X } ( b, a \otimes f ^{ \ast } y )
        \simeq
        \RHom _{ Y } ( y ^{ \vee }, f _{ \ast } \cRHom _{ X } ( b, a ) )
    \end{align}
    for all \( y \in \perf ( Y ) \) implies the vanishing of
    \(
    f _{ \ast } \cRHom _{ X } ( b, a )
    \), since
    \(
    {
        ( - )
    } ^{ \vee }
    \)
    is a (contravariant) autoequivalence of \( \perf ( Y ) \) and
    \( \perf ( Y ) \) generates \( \derived ^{ b } \coh Y \) (\cite[Theorem~3.1.1]{Bondal-van_den_Bergh}).
\end{proof}

\begin{definition}[{\(=\) the original version of~\cite[Definition~1.3]{2020arXiv201112743P}}]\label{definition:SSI}
    A noetherian scheme \( Y \) is said to be \emph{stably semiorthogonally indecomposable} (\emph{SSI} for short) if any proper morphism \( f \colon X \to Y \) from a regular scheme \( X \), any semiorthogonal decomposition of
    \(
       \derived
       (
        X
       )
    \)
    is \( f \)-linear.
\end{definition}

In~\cite{2020arXiv201112743P} Pirozhkov proved that a scheme which admits a finite morphism to an abelian variety is SSI\@. This paper heavily relies on this result.

\begin{remark}\label{remark:SSI vs NSSI}
    In fact~\cite[Definition~1.3]{2020arXiv201112743P} introduces a notion which is stronger than SSI\@. It is called \emph{NSSI} (noncommutatively SSI), and Pirozhkov proved that the aforementioned class of schemes belong to this category. However, the notion of SSI is sufficient for the purpose of this paper.
\end{remark}

\begin{remark}\label{remark:SSI implies SI}
    As shown in~\cite[Lemma~2.3]{2020arXiv201112743P}, if \( Y \) is a smooth proper variety and SSI, then it is SI\@.
\end{remark}

\begin{example}
    A noetherian affine scheme \( Y = \Spec R \) is SSI since \( R \in \perf (Y) \) is a classical generator (in the sense of~\cite{Bondal-van_den_Bergh}).
\end{example}

The following remark is obvious but plays a central role in this paper.
\begin{remark}\label{remark:transitivity of SI}
    Let
    \(
    f \colon X \to Y
    \)
    be as in~\cref{definition:semiorthogonal decomposition}. Suppose that \( Y \) is SSI and \( f \) is SI\@. Then \( X \) is SI\@.
\end{remark}

%
%------------------------------------------------------------------------
%

\section{Relative semiorthogonal indecomposability}

%
%------------------------------------------------------------------------
%

\subsection{Relative base locus}

\begin{definition}\label{definition:relative base locus}
    Let
    \(
    f \colon X \to Y
    \)
    be a morphism of algebraic spaces, and let \( L \) be an invertible sheaf on \( X \).
    The \emph{base locus of \( L \) relative to \( f \)} or the \emph{(relative) base locus of \( L \) over \( Y \)} is the closed subspace \( \baselocus _{ f } ( L ) \hookrightarrow X \) whose structure sheaf is defined as follows.
    \begin{align}\label{equation:defining sequence of Bsf(L)}
        f ^{ \ast } f _{ \ast } L \otimes _{ \cO _{ X } } L ^{ \vee } \to \cO _{ X }
        \to
        \cO _{ \baselocus _{ f } ( L ) }
        \to
        0
    \end{align}

    We call \( \baselocus _{ f } ( \omega _{ X } ) \) the \emph{(relative) canonical base locus of \( f \)} or the \emph{canonical base locus of \( X \) over \( Y \) (relative to \(f\))}.
\end{definition}

\begin{remark}\label{remark:relative base locus when target affine}
    If \( Y \) is affine in~\cref{definition:relative base locus}, then the first map of~\eqref{equation:defining sequence of Bsf(L)} is canonically isomorphic to the map
    \(
    H ^{ 0 } ( X, L ) \otimes L ^{ - 1 } \to \cO _{ X }
    \) and hence
    \(
    \baselocus _{ f } ( L )
    \)
    is nothing but the usual base locus
    \(
    \baselocus ( L )
    \).
\end{remark}

\begin{lemma}\label{lemma:invariance}
    In the situation of~\cref{definition:relative base locus}, the following equality of closed subspaces of \( X \) holds for any invertible sheaf
    \(
        M
    \)
    on \( Y \).

    \noindent
    \begin{align}
        \baselocus _{ f } ( L ) = \baselocus _{ f } ( L \otimes f ^{ \ast } M )
    \end{align}
\end{lemma}

\begin{proof}
    The projection formula
    
    \noindent
    \begin{align}\label{equation:projection formula}
        f
        _{
            \ast
        }
        \left(
            L
            \otimes
            f
            ^{
                \ast
            }
            M
        \right)
        \simeq
        f
        _{
            \ast
        }
        L
        \otimes
        M
    \end{align}
    implies
    
    \noindent
    \begin{align}
        f
        ^{
            \ast
        }       
        f
        _{
            \ast
        }
        \left(
            L
            \otimes
            f
            ^{
                \ast
            }
            M
        \right)
        \otimes
        \left(
            L
            \otimes
            f
            ^{
                \ast
            }
            M
        \right)
        ^{
            - 1
        }
        \stackrel{
            {\eqref{equation:projection formula}}
        }{
            \simeq
        }
        f
        ^{
            \ast
        }
        f
        _{
            \ast
        }
        L
        \otimes
        f
        ^{
            \ast
        }
        M
        \otimes        
        \left(
            L
            \otimes
            f
            ^{
                \ast
            }
            M
        \right)
        ^{
            - 1
        }        
        \simeq
        f
        ^{
            \ast
        }
        f
        _{
            \ast
        }
        L
        \otimes
        L
        ^{
            - 1
        }.
    \end{align}
    The assertion immediately follows from this.
\end{proof}

\begin{lemma}\label{lemma:flat base change}
    In the situation of~\cref{definition:relative base locus}, let
    \(
        y \colon Y ' \to Y
    \)
    be a flat morphism and consider the following Cartesian diagram.

    \noindent
    \begin{equation}\label{equation:cartesian diagram}
        \begin{tikzcd}
            X ' \arrow[r, "\ytilde"] \arrow[d, "f '"']
            &
            X \arrow[d,"f"]\\
            Y ' \arrow[r, "y"']
            &
            Y
        \end{tikzcd}
    \end{equation}
    
    Then
    \begin{align}
        \baselocus _{ f ' } ( \ytilde ^{ \ast } L )
        =
        \ytilde ^{ - 1 } \left( \baselocus _{ f } ( L ) \right).
    \end{align}
\end{lemma}

\begin{proof}
    It follows from the flatness of \( y \) that the base change map
    \begin{align}
        y ^{ \ast } f _{ \ast } L \to f ' _{ \ast } \ytilde ^{ \ast } L
    \end{align}
    is an isomorphism (\cite[Chapter~III~Proposition~9.3]{Hartshorne}). The assertion follows from this.
\end{proof}

\begin{lemma}\label{lemma:transitivity}
    In the situation of~\cref{definition:relative base locus}, suppose we have another morphism of algebraic spaces
    \(
       g
       \colon
       Y
       \to
       Z
    \).
    Then there is a canonical inclusion of closed subspaces of \( X \) as follows.

    \noindent
    \begin{align}\label{equation:comparison of relative base loci}
        \baselocus _{ g \circ f } ( L )
        \hookrightarrow
        \baselocus _{ f } ( L )
    \end{align}
    Moreover,~\eqref{equation:comparison of relative base loci} is an equality if \( g \) is an affine morphism.
\end{lemma}

\begin{proof}
    Note that there is a following commutative diagram as follows, where each arrow is the counit of the adjoint pair of functors next to it.

    \noindent
    \begin{equation}
        \begin{tikzcd}
            ( g \circ f ) ^{ \ast } ( g \circ f ) _{ \ast } L
            \arrow[r, equal]
            &
            f ^{ \ast } g ^{ \ast } g _{ \ast } f _{ \ast } L
            \arrow[r,"g ^{ \ast } \dashv g _{ \ast }"]
            \arrow[
                rd,
                "
                \left(
                    g
                    \circ
                    f
                \right) ^{ \ast }
                \dashv
                \left(
                    g
                    \circ
                    f
                \right) _{ \ast }
                "'
            ]
            &
            f ^{ \ast } f _{ \ast } L \arrow[d, "f ^{ \ast } \dashv f _{ \ast }"]\\
            &
            &
            L
        \end{tikzcd}
    \end{equation}
    The first assertion immediately follow from this.
    The second assertion follows from that the horizontal arrow is an epimorphism if
    \(
       g
    \)
    is affine.
\end{proof}

\begin{example}\label{example:empty relative base locus}
    Since
    \(
        \baselocus _{ \id } ( L )
        =
        \emptyset
    \)
    for the obvious reason,~\cref{lemma:transitivity} implies that
    \(
        \baselocus _{ g } ( L )
        =
        \emptyset
    \)
    for any affine (in particular finite) morphism \( g \).
\end{example}

In usual situation of birational geometry, there is an alternative description of the relative base locus as follows.
% Though we do not use it in this paper, we include it for possible future reference.

\begin{lemma}\label{lemma:alternative description of the relative base locus}
    In the situation of~\cref{definition:relative base locus}, suppose further that
    \(
       f
    \)
    is projective and that  \( Y \) is a quasi-projective variety.
    Then there is an alternative description for the relative base locus as follows.

    \noindent
    \begin{align}\label{equation:alternative description of the relative base locus}
        \cO _{ \baselocus _{ f } ( L ) }
        \simeq
        \coker
        \left(
        \bigoplus _{ \substack{M \in \Pic Y} } H ^{ 0 } ( X, L \otimes f ^{ \ast } M )
        \otimes _{ \bfk }
        {
            \left(
                L \otimes f ^{ \ast } M
             \right)
        } ^{ - 1 }
        \to
        \cO _{ X }
        \right)
    \end{align}
    Moreover, we have the following equality for any sufficiently ample invertible sheaf
    \(
       M
    \)
    on
    \(
       Y
    \).

    \noindent
    \begin{align}\label{equation:relative base locus is a base locus}
        \baselocus _{ f } ( L ) = \baselocus ( L \otimes f ^{ \ast } M )
    \end{align}
\end{lemma}

\begin{proof}
    To ease notation, let \( B \hookrightarrow X \) be the closed subscheme of \( X \) defined by the right hand side of~\eqref{equation:alternative description of the relative base locus} and let
    \(
       g
       \colon
       Y
       \to
       \Spec \bfk
    \)
    be the structure morphism of
    \(
       Y
    \).
    There is a following commutative diagram for any invertible sheaf \( M \in \Pic Y \).

    \noindent
    \begin{equation}
        \begin{tikzcd}
            f ^{ \ast } f _{ \ast } L \otimes L ^{ - 1 }
            \arrow[
                d,
                "\simeq",
                "\eqref{equation:projection formula}"'
            ]
            \arrow[dr] & \\
            f ^{ \ast } f _{ \ast } \left( L \otimes f ^{ \ast } M \right)
            \otimes
            \left( L \otimes f ^{ \ast } M \right) ^{ - 1 } \arrow[r]
            & \cO _{ X }\\
            H ^{ 0 } ( X, L \otimes f ^{ \ast } M )
            \otimes _{ \bfk }
            \left( L \otimes f ^{ \ast } M \right) ^{ - 1 }
            \arrow[
                u,
                "g ^{ \ast }
                \dashv
                g _{ \ast }"
            ]
            \arrow[ur] &
        \end{tikzcd}
    \end{equation}
    Thus we obtain the following commutative diagram, which implies that there is a natural morphism from the right hand side to the left hand side of~\eqref{equation:alternative description of the relative base locus}.

    \noindent
    \begin{equation}\label{equation:sum of counits}
        \begin{tikzcd}
            f ^{ \ast } f _{ \ast } L \otimes L ^{ - 1 } \arrow[dr] &\\
            \bigoplus _{ \substack{M \in \Pic Y} } H ^{ 0 } ( X, L \otimes f ^{ \ast } M )
            \otimes _{ \bfk }
            \left( L \otimes f ^{ \ast } M \right) ^{ - 1 }
            \arrow[u]
            \arrow[r]
            & \cO _{ X }
        \end{tikzcd}
    \end{equation}
    On the other hand, let \( M \) be sufficiently ample so that
    \(
        f _{ \ast } L \otimes M
    \)
    is a globally generated coherent sheaf on
    \(
       Y
    \). Then we obtain the epimorphism

    \noindent
    \begin{align}
        H ^{ 0 } ( Y, f _{ \ast } L \otimes M ) \otimes _{ \bfk } M ^{ - 1 }
        \twoheadrightarrow
        f _{ \ast } L,
    \end{align}
    from which we further obtain the epimorphism

    \noindent
    \begin{align}
        H ^{ 0 } ( Y, f _{ \ast } L \otimes M ) \otimes _{ \bfk }
        {
            \left(
                L \otimes f ^{ \ast } M
            \right)
        } ^{ - 1 }
        \twoheadrightarrow
        f ^{ \ast } f _{ \ast } L \otimes L ^{ - 1 }.
    \end{align}
    Thus we have confirmed the surjectivity of the vertical arrow of~\eqref{equation:sum of counits}, from which the assertion immediately follows.
\end{proof}

%
%------------------------------------------------------------------------
%

\subsection{The relative version of~\cite[]{2015arXiv150800682K}}

In this section we establish the relative version of~\cite[Theorem~3.1]{2015arXiv150800682K}, which constrains semiorthogonal decompositions in terms of the base locus of the canonical linear system. We prove that the relative base locus of the canonical bundle constrains relative semiorthogonal decompositions.

Let us first clarify what we mean by dualizing complexes and their relative base loci.

\begin{definition}\label{definition:dualizing complex}
    Let
    \(
       f
       \colon
       X
       \to
       Y
    \)
    be a proper morphism of finite dimensional connected noetherian schemes such that
    \(
       X
    \)
    is regular and
    \(
       Y
    \)
    admits a dualizing complex
    \(
       \omega
       _{
        Y
       }
       ^{
        \bullet
       }
       \in
       \derived
       ^{
        \bounded
       }
       \coh
       Y
    \)
    (see, say,~\cite[\href{https://stacks.math.columbia.edu/tag/0A87}{Tag~0A87}]{stacks-project} for the definition of dualizing complexes).
    Choose
    \(
        \omega
        _{
         Y
        }
        ^{
         \bullet
        }       
    \)
    such that

    \noindent
    \begin{align}
        \omega
        _{
         X
        }
        ^{
         \bullet
        }
        \coloneqq
        f
        ^{
            !
        }
        \omega
        _{
         Y
        }
        ^{
         \bullet
        }
    \end{align}
    is an invertible sheaf
    (\cite[\href{https://stacks.math.columbia.edu/tag/0AWV}{Tag~0AWV}]{stacks-project}). Then we call
    \(
       \omega
       _{
        X
       }
       \coloneqq
       \cH
       ^{
        0
       }
       \left(
        \omega
        _{
         X
        }
        ^{
         \bullet
        }
       \right)
    \)
    \emph{the} dualizing sheaf of
    \(
       X
    \).
    We let
    \(
       \baselocus
       _{
        f
       }
       \omega
       _{
        X
       }
    \)
    denote the relative base locus of
    \(
       \omega
       _{
        X
       }
    \).
\end{definition}

\begin{remark}\label{remark:relative canonical base locus}
    Though the dualizing sheaf
    \(
       \omega
       _{
        X
       }
    \)
    defined in~\cref{definition:dualizing complex} is unique only up to tensoring by invertible sheaves on
    \(
       Y
    \)
    (see~\cite[\href{https://stacks.math.columbia.edu/tag/0ATP}{Tag~0ATP}]{stacks-project}), the relative base locus
    \(
       \baselocus
       _{
        f
       }
       \omega
       _{
        X
       }
       \hookrightarrow
       X
    \)
    is well-defined by~\cref{lemma:invariance}.

    If
    \(
       X
    \)
    is smooth and proper over a field
    \(
       \bfk
    \), then we can choose
    
    \noindent
    \begin{align}
        \omega
        _{
         Y
        }
        ^{
         \bullet
        }
        \coloneqq
        \left(
         Y
         \to
         \Spec \bfk
        \right)
        ^{
         !
        }
        \bfk
        [
         - \dim X
        ]
    \end{align}
    so that
    \(
       \omega
       _{
        X
       }
       \simeq
       \Omega
       ^{
        \dim X
       }
       _{
        X
        /
        \bfk
       }
    \).
    If moreover
    \(
       \bfk
    \)
    is of characteristic
    \(
       0
    \),
    so that the resolution of singularity is guaranteed, then we can drop the properness assumption (the author owes the following argument to Alexander Kuznetsov). Indeed, we can compactify
    \(
       f
    \)
    into a morphism between proper schemes over
    \(
       \bfk
    \)
    such that the domain of the morphism is smooth over
    \(
       \bfk
    \). Then we can reduce the problem to the case where
    \(
       X
    \)
    is smooth and proper over
    \(
       \bfk
    \)
    by the base-change theorem for upper shriek~\cite[\href{https://stacks.math.columbia.edu/tag/0E9U}{Tag~0E9U}]{stacks-project}
    since open immersions are flat.

    Finally, note that the relative dualizing complex
    \(
       \omega
       ^{
        \bullet
       }
       _{
        X
        /
        \bfk
       }
    \)
    is different from a dualizing complex
    \(
        \omega
        ^{
         \bullet
        }
        _{
         X
        }       
    \)
    unless
    \(
       X
    \)
    is proper over
    \(
       \bfk
    \).
\end{remark}

\begin{theorem}\label{theorem:dichotomy}
    Let
    \(
        f \colon X \to Y
    \)
    be a proper morphism of finite dimensional noetherian algebraic spaces such that \( X \) is a regular connected scheme. Let
    \(
       \omega
       _{
        X
       }
    \)
    be the dualizing sheaf of
    \(
       X
    \)
    as in~\cref{definition:dualizing complex} (see also~\cref{remark:relative canonical base locus}) and
    
    \noindent
    \begin{align}\label{equation:an SOD}
        \derived ( X ) = \langle \cA, \cB \rangle
    \end{align}
    be an \( f \)-linear semiorthogonal decomposition. Then one of the followings holds.
    
    \noindent
    \begin{enumerate}[(A)]
        \item\label{item:points belong to A}
        For any closed point \( x \nin \baselocus _{ f } ( \omega _{ X } ) \) the sheaf \( \bfk ( x )\) is contained in \( \cA \).

        \item\label{item:points belong to B}
        For any closed point \( x \nin \baselocus _{ f } ( \omega _{ X } ) \) the sheaf \( \bfk ( x )\) is contained in \( \cB \).
    \end{enumerate}
    Moreover, when~\cref{item:points belong to A} (resp.~\cref{item:points belong to B}) holds, then the support of any object of \( \cB \) (resp. \( \cA \)) is contained in the closed subset
    \(
        \baselocus _{ f } ( \omega _{ X } )
    \).
\end{theorem}

\begin{proof}
    We follow the strategy of the proof of~\cite[Theorem~3.1]{2015arXiv150800682K}.
    Take a closed point
    \(
        x
        \nin
        \baselocus _{ f }
        (
            \omega _{ X }
        )
    \).
    We first prove that either
    \(
        \bfk ( x )
        \in
        \cA
    \)
    or
    \(
        \bfk ( x )
        \in
        \cB
    \)
    holds.

    Take an \'etale morphism
    
    \noindent
    \begin{align}\label{equation:etale chart}
        y \colon Y ' \to Y
    \end{align}
    from a noetherian affine scheme \( Y ' \) and consider the Cartesian diagram as in~\eqref{equation:cartesian diagram}. Choose
    \(
       y
    \)
    so that
    \begin{itemize}
        \item
        \(
           f ( x )
           \in
           \Image ( y )
        \)
        \item
        \(
           |
           y ^{ - 1 }
           \left(
            f ( x )
           \right)
           |
           =
           1
        \).
    \end{itemize}
    Let
    \(
        x '
        \in
        X '
    \)
    be the unique point of
    \(
       X '
    \)
    such that
    \(
       \ytilde
       (
        x '
       )
       =
       x
    \).

    The flatness of \( y \), the isomorphism
    \(
       \ytilde
       ^{
        \ast
       }
       \omega _{ X }
       \simeq
       \omega
       _{
        X '
       }
    \)
    (by the base-change theorem~\cite[\href{https://stacks.math.columbia.edu/tag/0E9U}{Tag~0E9U}]{stacks-project}
    as
    \(
       y
    \)
    is \'etale), and~\cref{lemma:flat base change} imply that
    \(
        x '
        \nin
        \baselocus _{ f ' }
        (
            \omega _{ X ' }
        )
        \subseteq
        X '
    \).
    Then by~\cref{remark:relative base locus when target affine}, there is a section
    \(
        s \in H ^{ 0 }
        (
            X ',
            \omega
            _{
                X '
            }
        )
    \)
    such that
    \(
        s ( x )
        \neq
        0
    \).
    % Let

    % \noindent
    % \begin{align}\label{equation:base change SOD}
    %     \derived ( X ' ) = \langle \cA ', \cB ' \rangle
    % \end{align}
    % be the base change of the \( f \)-linear semiorthogonal decomposition~\eqref{equation:an SOD} by the morphism \( y \), as established in~\cite{MR2801403}. We prove that
    % \(
    %     \bfk ( x )
    % \)
    % is contained in either
    % \(
    %     \cA '
    % \)
    % or
    % \(
    %     \cB '
    % \).

    Consider the decomposition of
    \(
        \bfk ( x )
    \)
    with respect to the
    \(
       f
    \)-linear semiorthogonal decomposition~\eqref{equation:an SOD}:

    \noindent
    \begin{align}\label{equation:decomposition of k(x)}
        b
        \xrightarrow[]{ \xi }
        \bfk ( x )
        \xrightarrow[]{ \eta }
        a
        \xrightarrow[  ]{ \zeta }
        b [ 1 ],
    \end{align}
    where
    \(
        a \in \cA
    \)
    and
    \(
        b \in \cB
    \).
    Apply the functor
    \(
       \ytilde
       ^{
        \ast
       }
    \)
    to~\eqref{equation:decomposition of k(x)} to obtain the following distinguished triangle, where
    \(
       b '
       \coloneqq
       \ytilde
       ^{
           \ast
       }
       b,
       \xi '
       \coloneqq
       \ytilde
       ^{
           \ast
       }
       \xi
    \)
    and so on.
    
    \noindent
    \begin{align}
        b '
        \xrightarrow[]{
            \xi '
            }
        \ytilde
        ^{
            \ast
        }
        \bfk ( x )
        =
        \bfk
        (
            x '
        )
        \xrightarrow[]{
            \eta '
            }
        a '
        \xrightarrow[  ]{
            \zeta '
            }
        b ' [ 1 ],        
    \end{align}
    Then it follows that

    \noindent
    \begin{align}
        \left[
            a '
            \xrightarrow[]{
                \zeta '
                }
            b ' [ 1 ]
            \xrightarrow[]{
                \otimes s
                }
            b '
            \otimes
            \omega _{ X ' } [ 1 ]
        \right]
        &
        \in
        \RHom _{ X ' }
        \left(
            a ',
            b ' \otimes \omega _{ X ' } [ 1 ]
        \right)
        \\
        &
        \stackrel{
            \textrm{
                \(
                   Y '
                \)
                is affine
            }
        }{
            =
        }
        f '
        _{
            \ast
        }
        \cRHom
        _{
            X '
        }
        \left(
            a ',
            b ' \otimes \omega _{ X ' } [ 1 ]
        \right)
        \\
        &
        \stackrel{\text{\cref{corollary:Grothendieck duality}}}{ \simeq }
        {
            \left(
                f '
                _{
                    \ast
                }
                \cRHom _{
                    X '
                }
                \left(
                    b '
                    \otimes \omega _{ X ' } [ 1 ],
                    a '
                    \otimes \omega _{ X ' } ^{ \bullet }
                \right)
            \right)
        } ^{ \vee }
        \\
        &
        \stackrel{
            \textrm{
                \cref{definition:dualizing complex}
            }
        }{
            \simeq
        }
        {
            \left(
                f '
                _{
                    \ast
                }
                \cRHom _{
                    X '
                }
                \left(
                    b ',
                    a '
                \right)
            \right)
        } ^{ \vee }
        [
            1
        ],\label{equation:composition with s}
    \end{align}
    where

    \noindent
    \begin{align}
        {
            \left( ? \right)
        } ^{ \vee }
        \coloneqq
        \cRHom _{ Y ' } ( ?, \omega _{ Y ' } ^{ \bullet } )
        \colon
        {
            \derived ( Y ' )
        } ^{ \op } \simto \derived ( Y ' ).
    \end{align}
    Then~\eqref{equation:composition with s}
    \(
        =
        0
    \)
    as
    
    \noindent
    \begin{align}
        f '
        _{
            \ast
        }
        \cRHom _{
            X '
        }
        \left(
            b ',
            a '
        \right)
        &
        =
        f '
        _{
            \ast
        }
        \cRHom _{
            X '
        }
        \left(
            \ytilde
            ^{
                \ast
            }
            b,
            \ytilde
            ^{
                \ast
            }
            a
        \right)
        \\
        &
        \simeq
        f '
        _{
            \ast
        }
        \ytilde
        ^{
            \ast
        }
        \cRHom _{
            X
        }
        \left(
            b,
            a
        \right)
        \\
        &
        \stackrel{
            \textrm{\cite[Chapter~III~Proposition~9.3]{Hartshorne}}
        }{
            \simeq
        }
        y
        ^{
            \ast
        }
        f
        _{
            \ast
        }
        \cRHom _{
            X
        }
        \left(
            b,
            a
        \right)        
        \\
        &
        =
        y ^{ \ast }
        0
        \\
        &
        =
        0.
    \end{align}

    On the other hand, multiplication by the section \( s \) is an isomorphism on its non-vanishing locus
    \(
        U
        \coloneqq
        X ' _{ s }
        \subseteq X '
    \). Thus we conclude that

    \noindent
    \begin{align}
        \zeta \vert _{ U } = 0.
    \end{align}
    It follows from this and~\eqref{equation:decomposition of k(x)} that either
    \(
        ( a ' \vert _{ U }, b ' \vert _{ U } )
        \simeq
        ( 0, \bfk ( x ) )
    \)
    or
    \(
        ( \bfk ( x ), 0 )
    \).
    This implies that either
    \(
       a
    \)
    or
    \(
       b
    \)
    is
    \(
       0
    \)
    on an open neighborhood of
    \(
       x
    \)
    in
    \(
       X
    \).
    If the former is the case, then it implies that
    \(
       \eta
       =
       0
    \), so that
    
    \noindent
    \begin{align}
        b
        \simeq
        \bfk ( x )
        \oplus
        a [ - 1 ].
    \end{align}
    As
    \(
       \RHom
       _{
        X
       }
       \left(
        b,
        a [ - 1 ]
       \right)
       =
       0
    \)
    by the semiorthogonality, this implies that
    \(
       a = 0
    \)
    and
    \(
       b
       \simeq
       \bfk ( x )
    \).
    Similarly, if the latter is the case, then it follows that
    \(
       b = 0
    \)
    and
    \(
       a
       \simeq
       \bfk ( x )
    \).

    Next, the connectedness of
    \(
       X
    \)
    implies that it is impossible to have two closed points
    \(
       x
       _{
        1
       }
       \neq
       x
       _{
        2
       }
       \in
       X
    \)
    such that
    \(
       \bfk
       (
        x
        _{
            1
        }
       )
       \in
       \cA
    \)
    and
    \(
       \bfk
       (
        x
        _{
            2
        }
       )
       \in
       \cB
    \).
    Indeed, if this happens, then it follows that the support of objects of
    \(
       \cA
    \)
    and
    \(
       \cB
    \)
    are closed subsets of
    \(
       X
    \)
    which do not contain the generic point of
    \(
       X
    \).
    On the other hand, there is a triangle
    
    \noindent
    \begin{align}
        b
        \to
        \cO
        _{
            X
        }
        \to
        a
        \to
        b [ 1 ]
    \end{align}
    obtained from~\eqref{equation:an SOD}. Then
    
    \noindent
    \begin{align}
        X
        =
        \Supp ( \cO _{ X } )
        \subseteq
        \Supp ( a )
        \cup
        \Supp ( b )
        \neq
        X,
    \end{align}
    a contradiction. Thus we obtain the first assertion.

    Now assume that the case~\cref{item:points belong to A} holds, and take an arbitrary object
    \(
       b
       \in
       \cB
    \). Suppose for a contradiction that there exits a closed point
    \(
       x
       \in
       \Supp ( b )
       \setminus
       \baselocus
       _{
        f
       }
       (
        \omega
        _{
            X
        }
       )
    \). Let
    
    \noindent
    \begin{align}
        j
        \colon
        \left\{
           x
        \right\}
        \hookrightarrow
        X        
    \end{align}
    be the embedding of the point. The derived version of the Nakayama's lemma (see, say,~\cite[Lemma~2.3]{2002.03303v1}) implies that
    \(
       j
       ^{
        \ast
       }
       b
       \neq
       0
    \), which in turn implies that there exits a non-trivial morphism
    
    \noindent
    \begin{align}
        b
        \to
        j _{ \ast }
        j ^{ \ast }
        b.  
    \end{align}
    Hence we see that
    
    \noindent
    \begin{align}
        f
        _{
            \ast
        }
        \cRHom
        _{
            X
        }
        \left(
            b,
            j _{ \ast }
            j ^{ \ast }
            b
        \right)
        \neq
        0.
    \end{align}
    On the other hand, the assumption~\cref{item:points belong to A} implies
    \(
       \bfk ( x )
       \in
       \cA
    \)
    and hence
    \(
        j _{ \ast }
        j ^{ \ast }
        b
        \in
        \cA
    \). Thus we obtain a contradiction.

    Finally, if the case~\cref{item:points belong to B} holds, then we can similarly prove the assertion by contradiction by using~\cref{corollary:Grothendieck duality}.
\end{proof}

\begin{remark}\label{remark:comparison}
    \cref{theorem:dichotomy} is also shown in~\cite[Theorem~2.3]{2023arXiv230106818L} in a greater generality. Here we briefly compare this paper with~\cite[]{2023arXiv230106818L}.

    First of all, the definition of the relative canonical base locus given in~\cite[Definition~1.2]{2023arXiv230106818L} (\emph{base locus of \(f\)} in their terminology and denoted by \(\baselocus |f|\)) is different from the one we gave in~\cref{definition:relative base locus}. They coincide when
    \(
       Y
    \)
    is the spectrum of a field (see the end of~\cite[Section~1.2]{2023arXiv230106818L}). They also introduced a slightly modified version
    \(
       \baselocus '
       | f |
    \)
    with the following properties (see~\cite[Example~1.3~(2)]{2023arXiv230106818L} for details).
    \begin{itemize}
        \item 
        \(
            \baselocus '
            | f |
            \subseteq           
            \baselocus
            | f |
        \), and they coincide if
        \(
           f
        \)
        is of finite homological dimension.

        \item
        \(
            \baselocus '
            | f |
        \)
        is local on the base; i.e., for any open embedding
        \(
           U
           \hookrightarrow
           Y
        \)
        one has
        \(
            \baselocus '
            | f |
            \cap
            f
            ^{
                - 1
            }
            (
                U
            )
            =
            \baselocus '
            | f \vert _{ U } |
        \).

        \item
        The main results of~\cite[]{2023arXiv230106818L} are valid if one replaces
        \(
           \baselocus | f |
        \)
        with
        \(
           \baselocus '
           | f |
        \).
    \end{itemize}
    Since
    \(
       \baselocus
       _{
        f
       }
       \omega _{ X }
    \)
    is also local on the base, and the inclusion
    
    \noindent
    \begin{align}\label{equation:two definitions}
        \baselocus ' | f |
        \subseteq
        \baselocus
        _{
         f
        }
        \omega _{ X }
    \end{align}
    is obvious when
    \(
       Y
    \)
    is affine, it follows that the inclusion~\eqref{equation:two definitions} holds true in general. Hence~\cite[Theorem~2.3]{2023arXiv230106818L} implies~\cref{theorem:dichotomy}.
    It is not clear to the author whether~\eqref{equation:two definitions} is always an equality or not.

    Though our base locus is (at least a priori) bigger than theirs, it is simpler and more natural for birational geometry; these features are exploited in~\cref{section:Semiorthogonal indecomposability of irregular minimal surfaces}. On the other hand~\cite[Definition~1.2]{2023arXiv230106818L} works for singular
    \(
       X
    \)
    as well, for which the dualizing complex is not necessarily a shift of an invertible sheaf. They prove~\cite[Theorem~2.3]{2023arXiv230106818L} by introducing a well-tailored spanning class (see also~\cite[]{MR4395364} for the case of Cohen-Macaulay varieties).
    It is not clear how to generalize the arguments of this paper to singular
    \(
       X
    \).

    From~\cite[Theorem~2.3]{2023arXiv230106818L} they obtain a couple of indecomposability results, which, for (proper morphisms from) regular schemes, we can also obtain from~\cref{theorem:dichotomy}.
\end{remark}

%
%------------------------------------------------------------------------
%
\subsection{Applications to semiorthogonal indecomposability}

\cref{example:PBs as relative canonical base locus} below motivated this work.

\begin{example}\label{example:PBs as relative canonical base locus}
    The paracanonical base locus \( \parabaselocus ( X )\) of \( X \) is defined as follows (\cite[Corollary~3.5]{2021arXiv210709564L}).

    \noindent
    \begin{align}
        \cO _{ \parabaselocus ( X ) }
        \simeq
        \coker
        \left(
        \bigoplus _{ \substack{M ' \in \Pic ^{ 0 } X} } H ^{ 0 } ( X, \omega _{ X } \otimes M '  )
        \otimes _{ \bfk }
        {
            \left(
                \omega _{ X } \otimes M '
            \right)
        } ^{ - 1 }
        \to
        \cO _{ X }
        \right)
    \end{align}
    Put
    \(
       f
       =
       \albanese
       _{
        X
       }
       \colon
       X
       \to
       Y
       =
       \Albanese
       _{
        X
       }
    \),
    the albanese morphism of \( X \). Then the natural isomorphism
    \(
       \Pic
       ^{
        0
       }
       (
        Y
       )
       =
       \Pic
       ^{
        0
       }
       (
        \Albanese
        _{
            X
        }
       )
       \simeq
       \Pic ^{ 0 }
       (
        X
       )
    \)
    and~\cref{lemma:alternative description of the relative base locus} imply the obvious quotient map
    \(
        \cO _{ \parabaselocus ( X ) }
        \twoheadrightarrow
        \cO _{ \baselocus _{ f } ( \omega _{ X } ) }
    \); i.e., a closed immersion
    \(
        \baselocus _{ f } ( \omega _{ X } )
        \hookrightarrow
        \parabaselocus ( X )
    \).
    It is shown in~\cite[Theorem~1.5]{2021arXiv211006795C} that this actually is an isomorphism.

    In~\cite[Theorem~4.5]{2021arXiv210709564L} it is shown that the paracanonical base locus is empty for the
    \(
       k
    \)-th symmetric product of a smooth projective curve of genus
    \(
       g
       \ge
       2
    \)
    for
    \(
       k
       \le
       g - 1
    \).
    On the contrary, the canonical base locus of such a variety is not necessarily empty and hence we can not conclude the semiorthogonal indecomposability by using~\cite[Theorem~1.2]{2015arXiv150800682K}.

    Xun Lin proved in~\cite[Theorem~1.4]{2021arXiv210709564L} that one can replace the canonical base locus in~\cite[Theorem~1.2]{2015arXiv150800682K} with the paracanonical base locus, by bootstrapping the arguments of~\cite[]{2015arXiv150800682K}. From our point of view, the same result is obtained by applying~\cref{theorem:dichotomy} to the Albanese morphism and then by~\cref{theorem:Pirozhkov}. In any case, this implies the semiorthogonal indecomposability of aforementioned varieties by the aforementioned emptiness.

    It is worth noting that the technique of~\cite[]{2021arXiv210709564L} applies to singular
    \(
       X
    \)
    as well, whereas
    \(
       \Pic ^{ 0 } ( X )
    \)
    for such
    \(
       X
    \)
    may not be an abelian variety and hence the arguments of this paper do not apply. In this sense the former is more flexible than the latter. In fact, in the author's appendix to~\cite[]{2023arXiv230106818L}, he uses the technique to prove that semistable curves are semiorthogonally indecomposable.
\end{example}

Even if the relative canonical base locus is not empty, in good cases we can still conclude the relative semiorthogonal indecomposability.

\begin{corollary}\label{corollary:no relative SOD}
    In the situation of~\cref{theorem:dichotomy}, suppose that for each connected component of
    \(
        \baselocus _{ f }
        (
            \omega _{ X }
        )
    \)
    there exists a Zariski open neighborhood in
    \(
       X
    \)
    on which
    \(
        \omega _{ X }
    \)
    is isomorphic to the pullback of an invertible sheaf on
    \(
       Y
    \)
    (i.e.,
    \(
       f
    \)-trivial).
    Then \( X \) admits no \( f \)-linear semiorthogonal decomposition.
\end{corollary}

\begin{proof}
    It is enough to show that any \( f \)-linear semiorthogonal decomposition
    \(
        \langle
            \cA,
            \cB
        \rangle
    \)
    is an orthogonal decomposition. Take
    \(
        a
        \in
        \cA
    \)
    and
    \(
        b
        \in
        \cB
    \).
    Suppose that~\cref{theorem:dichotomy}~\cref{item:points belong to A} holds, so that \( a \) is supported in
    \(
        \baselocus _{ f }
        (
            \omega _{ X }
        )
    \). Then it follows from the assumption of~\cref{corollary:no relative SOD} and the \( f \)-linearity that
    \begin{align}\label{equation:stable under Serre}
        a
        \otimes
        \omega _{ X }
        ^{ \bullet }
        \in
        \cA.
    \end{align}
    Hence

    \noindent
    \begin{align}
        f _{ \ast }
        \cRHom _{ X }
        (
            a,
            b
        )
        \stackrel{
            \text{~\cref{corollary:Grothendieck duality} }
            }{
                \simeq
            }
        {
            f _{ \ast }
            \cRHom _{ X }
            \left(
                b,
                a \otimes \omega _{ X } ^{ \bullet }
            \right)
        } ^{ \vee } = 0.
    \end{align}
    The last equality follows from the vanishing
    \(
        f _{ \ast }
        \cRHom _{ X }
        (
            b,
            a
            \otimes
            \omega _{ X }
            ^{ \bullet }
        )
        =
        0
    \), which follows from~\eqref{equation:stable under Serre}.
    Thus we have confirmed that the decomposition is actually orthogonal.
    If we assume~\cref{theorem:dichotomy}~\cref{item:points belong to B} instead, then similarly it follows that
    \(
        b
        \otimes
        \omega _{ X }
        ^{ - 1 }
        \in
        \cB
    \)
    and one can confirm
    \(
        f
        _{
            \ast
        }
        \cRHom
        _{
            X
        }
        (
            a,
            b
        )
        =
        0
    \)
    by similar arguments as above.
\end{proof}

\begin{corollary}\label{corollary:transitivity of SI}
    Let \( f \colon X \to Y \) be a morphism of smooth projective varieties such that \( Y \) is SSI\@. If moreover

    \noindent
    \begin{align}
        |
        \baselocus _{ f }
        (
            \omega _{ X }
        )
        |
        <
        \infty,
    \end{align}
    then \( X \) admits no semiorthogonal decomposition.
\end{corollary}

\begin{proof}
    Immediately follows from~\cref{corollary:no relative SOD} and~\cref{remark:transitivity of SI}.
\end{proof}

\begin{lemma}\label{lemma:any SOD is f-linear}
    In the situation of~\cref{theorem:dichotomy}, suppose further that there exists a closed subset
    \(
       Z
       \hookrightarrow
       X
    \)
    such that    
    
    \noindent
    \begin{align}
        \dim
        f
        \left(
            Z
        \right)
        \le
        0
        \left(
            \iff
            \left|
                f
                \left(
                    Z
                \right)
            \right|
            <
            \infty
        \right).
    \end{align}
    Then any SOD
    
    \noindent
    \begin{align}
        \derived
        (
            X
        )
        =
        \langle
            \cA,
            \cB
        \rangle
    \end{align}
    such that
    
    \noindent
    \begin{align}
        \Supp
        \left(
            a
        \right)
        \subseteq
        Z
    \end{align}
    for any
    \(
       a
       \in
       \cA
    \)
    is
    \(
       f
    \)-linear.
\end{lemma}

\begin{proof}
    The assertion is trivial if
    \(
       Y
    \)
    is a scheme, but we find the assertion for algebraic spaces not completely trivial. As we could not find a proof in the literature, we include a proof below.

    Take objects
    \(
       a
       \in
       \cA
    \)
    and
    \(
       b
       \in
       \cB
    \).
    By~\cite[Lemma~7.40]{MR2434186}, for some
    \(
       m > 0
    \)
    there exists a bounded complex of coherent sheaves
    \(
       \abar
    \)
    on
    \(
       m Z
    \)
    such that
    \(
       a
       \simeq
       \iota
       _{
        \ast
       }
       \abar
    \),
    where
    \(
       \iota
       \colon
       m Z
       \hookrightarrow
       X
    \)
    is the immersion morphism. Then we have
    
    \noindent
    \begin{align}
        f
        _{
            \ast
        }
        \cHom
        _{
            X
        }
        \left(
            b,
            a
        \right)
        \simeq
        f
        _{
            \ast
        }
        \iota        
        _{
            \ast
        }
        \cHom
        _{
            m Z
        }
        \left(
            \iota
            ^{
                \ast
            }
            b,
            \abar
        \right).
    \end{align}

    On the other hand, by the assumption there exists a closed immersion from a
    \(
       0
    \)-dimensional algebraic space
    \(
       j
       \colon
       W
       \hookrightarrow
       Y
    \)
    such that
    \(
       f
       \circ
       \iota
    \)
    decomposes into a morphism
    \(
        g
        \colon
       m Z
       \to
       W
    \)
    and
    \(
       j
    \),
    so that
    
    \noindent
    \begin{align}
        f
        _{
            \ast
        }
        \iota        
        _{
            \ast
        }
        \cHom
        _{
            m Z
        }
        \left(
            \iota
            ^{
                \ast
            }
            b,
            \abar
        \right)
        \simeq        
        j
        _{
            \ast
        }
        g
        _{
            \ast
        }
        \cHom
        _{
            m Z
        }
        \left(
            \iota
            ^{
                \ast
            }
            b,
            \abar
        \right).
    \end{align}
    As is known well, since
    \(
       \dim W = 0
    \),
    \(
       W
    \)
    is indeed an affine scheme. Hence
    
    \noindent
    \begin{align}
        g
        _{
            \ast
        }
        \cHom
        _{
            m Z
        }
        \left(
            \iota
            ^{
                \ast
            }
            b,
            \abar
        \right)
        =
        \Hom
        _{
            m Z
        }
        \left(
            \iota
            ^{
                \ast
            }
            b,
            \abar
        \right)
        \simeq
        \Hom
        _{
            X
        }
        \left(
            b,
            \iota
            _{
                \ast
            }
            \abar
            =
            a
        \right)
        =
        0.
    \end{align}
    This concludes the proof.
\end{proof}

% \begin{example}
%     A typical application of~\cref{corollary:transitivity of SI} is to the albanese morphism of a smooth projective variety \( X \), where \( Y \) is the albanese variety of \( X \), which is known to be SSI by~\cite[Theorem~1.4]{2020arXiv201112743P}. By~\cref{corollary:transitivity of SI}, if the relative canonical base locus of \( f \), which coincides with \( \parabaselocus ( \omega _{ X } )\) as we recalled in~\cref{example:PBs as relative canonical base locus}, is finite, then \( X \) admits no semiorthogonal decomposition.
% \end{example}

%
%------------------------------------------------------------------------
%
\section{Semiorthogonal indecomposability of irregular minimal surfaces}\label{section:Semiorthogonal indecomposability of irregular minimal surfaces}

Throughout~\cref{section:Semiorthogonal indecomposability of irregular minimal surfaces} we work over an algebraically closed field \( \bfk \) of an arbitrary characteristic.

\begin{theorem}\label{theorem:irregular surface is SI}
    Let
    \(
       X
    \)
    be a minimal surface with
    \(
       H
       ^{
        1
       }
       \left(
        X,
        \cO
        _{
            X
        }
       \right)
       \neq
       0
    \).
    Then
    \(
       \derived
       (
        X
       )
    \)
    admits no non-trivial semiorthogonal decompositions.
\end{theorem}

In the proof of~\cref{theorem:irregular surface is SI}, we use the following results on relative canonical base locus due to Kazuhiro Konno.

\begin{theorem}[{\(=\)\cite[Theorem~2.4]{MR2481686}}]\label{theorem:birational Konno}
    Let
    \(
       f
       \colon
       X
       \to
       Y
    \)
    be a birational projective morphism of normal algebraic spaces of dimension
    \(
       2
    \)
    such that
    \(
       X
    \)
    is a nonsingular scheme and
    \(
       \omega
       _{
        X
       }
    \)
    is
    \(
       f
    \)-nef. Then there exits a factorization
    
    \noindent
    \begin{equation}\label{equation:birational Konno}
        \begin{tikzcd}
            X
            \arrow[r, "g"]
            \arrow[rr, bend right, "f"']
            &
            X '
            \arrow[r]
            &
            Y
        \end{tikzcd}
    \end{equation}
    such that
    \(
       X '
    \)
    has only rational singularities and
    
    \noindent
    \begin{align}
        \left |
        \baselocus
        _{
            f
        }
        \omega
        _{
            X
        }
        \setminus
        \exceptionalset
        (
            g
        )
        \right |
        <
        \infty.
    \end{align}
\end{theorem}

\begin{theorem}\label{theorem:base point free over rational singularities}
    In the setting of~\cref{theorem:birational Konno}, assume further that
    \(
       Y
    \)
    has only rational singularities. Then
    \(
       \baselocus
       _{
        f
       }
       \omega
       _{
        X
       }
       =
       \emptyset
    \).
\end{theorem}

\begin{proof}
    See the paragraph before~\cite[Proposition~2.7]{MR2481686}.
\end{proof}

\begin{theorem}[{\(=\)\cite[Theorem~3.1]{MR2654305}}]\label{theorem:original fibered Konno}
    Let
    
    \noindent
    \begin{align}
        h
        \colon
        X
        \to
        C
    \end{align}
    be a projective morphism of schemes of connected fibers where
    \(
       X
    \)
    is a nonsingular surface,
    \(
       C
    \)
    is a nonsingular curve, and
    \(
       \omega
       _{
        X
       }
    \)
    is
    \(
       h
    \)-nef.
    Suppose further that the generic fiber of
    \(
       h
    \)
    is of genus
    \(
       \ge
       2
    \).
    Let
    \(
       F
       =
       X
       _{
        c
       }
       \hookrightarrow
       X
    \)
    be the schematic fiber of a closed point
    \(
       c
       \in
       C
    \).
    Then
    \(
       \baselocus
       \omega
       _{
        F
       }
    \)
    supports at most exceptional sets of rational singularities or weak elliptic singularities.    
\end{theorem}

\begin{lemma}\label{lemma:baselocus commute with taking fibers}
    In the situation of~\cref{theorem:original fibered Konno}, we have
    
    \noindent
    \begin{align}
        F
        \cap
        \baselocus
        _{
            h
        }
        \left(
            \omega
            _{
                h
            }
        \right)
        \simeq
        \baselocus
        \omega
        _{
            F
        }.
    \end{align}
\end{lemma}

\begin{proof}
    Immediately follows from the following isomorphisms.
    
    \noindent
    \begin{align}
        \left(
            h
            _{
                \ast
            }
            \omega
            _{
                h
            }
        \right)
        \vert
        _{
            c
        }
        \simeq
        H ^{ 0 }
        \left(
            F,
            \omega
            _{
                h
            }
            \vert
            _{
                F
            }
        \right)
        \simeq
        H ^{ 0 }
        \left(
            F,
            \omega
            _{
                F
            }
        \right),
    \end{align}
    where the first isomorphism follows from the cohomology and base change~\cite[Chapter~III~Theorem~12.11]{Hartshorne} and the relative duality.
\end{proof}

\begin{corollary}\label{corollary:fibered Konno}
    In the situation of~\cref{theorem:original fibered Konno}, there exists a factorization
    
    \noindent
    \begin{equation}\label{equation:fibered Konno}
        \begin{tikzcd}
           X
           \arrow[r, "f"]
           \arrow[rr, bend right, "h"']
           &
           Y
           \arrow[r]
           &
           C
        \end{tikzcd}
    \end{equation}
    such that
    \(
       Y
    \)
    is a normal algebraic space,
    \(
       f
    \)
    is a birational projective morphism, and
    
    \noindent
    \begin{align}
        \left |
            \baselocus
            _{
                h
            }
            \omega
            _{
                X
            }
            \setminus
            \exceptionalset
            (
                f
            )
        \right |
        <
        \infty
    \end{align}
\end{corollary}

\begin{proof}
    Note first that
    
    \noindent
    \begin{align}
        \baselocus
        _{
            h
        }
        \omega
        _{
            X
        }
        =
        \baselocus
        _{
            h
        }
        \omega
        _{
            h
        }
    \end{align}
    by~\cref{lemma:invariance}. Then~\cref{lemma:baselocus commute with taking fibers} immediately implies the following equality.
    
    \noindent
    \begin{align}
        \baselocus
        _{
            h
        }
        \omega
        _{
            h
        }
        =
        \bigcup        
        _{
            F
            \colon
            \mathrm{singular}
        }
        \baselocus
        \omega
        _{
            F
        }
    \end{align}
    Then by~\cref{theorem:original fibered Konno}, there is a birational contraction
    \(
       f
       \colon
       X
       \to
       Y
    \)
    to an algebraic fiber space satisfying the assertions of~\cref{corollary:fibered Konno}.
\end{proof}

\begin{proof}[Proof~of~\cref{theorem:irregular surface is SI}]
    By the assumption, we have
    \(
       \dim
       \albanese
       _{
        X
       }
       \left(
        X
       \right)
       \in
       \left \{
          1,
          2
       \right \}
    \).

    Suppose first that
    \(
        \dim
        \albanese
        _{
         X
        }
        \left(
         X
        \right)
       =
       2
    \).
    Take the Stein factorization of the Albanese morphism of
    \(
       X
    \)
    as follows.
    
    \noindent
    \begin{equation}\label{equation:Stein factorization of Albanese, irregularity = 2}
        \begin{tikzcd}
           X
           \arrow[
            rr,
            "\albanese _{ X }"            
           ]
           \arrow[
            dr,
            "f"'
           ]
           &
           &
           \Albanese
           _{
            X
           }
           \\
           &
           Y
           \arrow[
            ur
           ]
           &
        \end{tikzcd}
    \end{equation}
    
    % \noindent
    % \begin{align}\label{equation:Stein factorization of Albanese, irregularity = 2}
    %     f
    %     \colon
    %     X
    %     \to
    %     Y
    % \end{align}
    % be the Stein factorization of the Albanese morphism.
    The assumption implies
    \(
       \dim
       Y
       =
       2
    \),
    and
    \(
       Y
    \)
    is SSI by~\cref{theorem:Pirozhkov}. Hence it is enough to show that
    \(
       f
    \)
    is SI\@. Now we can apply~\cref{theorem:birational Konno} to the morphism
    \(
       f
    \)
    of~\eqref{equation:Stein factorization of Albanese, irregularity = 2} to obtain the factorization as in~\eqref{equation:birational Konno}.

    Take an arbitrary
    \(
       f
    \)-linear semiorthogonal decomposition of
    \(
       \derived
       (
        X
       )
    \).
    Then it follows from~\cref{theorem:birational Konno,lemma:any SOD is f-linear,theorem:dichotomy} that the semiorthogonal decomposition under consideration is indeed
    \(
       g
    \)-linear. Then by applying~\cref{theorem:base point free over rational singularities} to the morphism
    \(
       g
    \)
    and~\cref{theorem:dichotomy} again, we conclude that the semiorthogonal decomposition is trivial.

    Suppose next that
    \(
       \dim
       \albanese
       _{
        X
       }
       \left(
        X
       \right)
       =
       1
    \).
    Take the Stein factorization of the Albanese morphism of
    \(
       X
    \)
    as follows.
    
    \noindent
    \begin{equation}\label{equation:Stein factorization of Albanese, irregularity = 1}
        \begin{tikzcd}
           X
           \arrow[
            rr,
            "\albanese _{ X }"            
           ]
           \arrow[
            dr,
            "h"'
           ]
           &
           &
           \Albanese
           _{
            X
           }
           \\
           &
           C
           \arrow[
            ur
           ]
           &
        \end{tikzcd}
    \end{equation}
    
    % \noindent
    % \begin{align}\label{equation:Stein factorization of Albanese, irregularity = 1}
    %     h
    %     \colon
    %     X
    %     \to
    %     C
    % \end{align}
    % be the Stein factorization of the Albanese morphism.
    As in the previous case, the assumption implies that
    \(
       \dim
       C
       =
       1
    \)
    and
    \(
       C
    \)
    is SSI by~\cref{theorem:Pirozhkov}, so it is enough to show that
    \(
       h
    \)
    is SI\@.
    Now we can apply~\cref{corollary:fibered Konno} to the morphism
    \(
       h
    \)
    of~\eqref{equation:Stein factorization of Albanese, irregularity = 1} to obtain the factorization as in~\eqref{equation:fibered Konno}.

    Take an arbitrary
    \(
       h
    \)-linear semiorthogonal decomposition of
    \(
       \derived
       (
        X
       )
    \).
    Assume first that the general fiber of
    \(
       h
    \)
    is of genus
    \(
       \ge
       2
    \).
    Then it follows from~\cref{corollary:fibered Konno,lemma:any SOD is f-linear,theorem:dichotomy} that the semiorthogonal decomposition under consideration is indeed
    \(
       f
    \)-linear. Then by applying the arguments for the case where
    \(
       \dim
       \albanese
       _{
        X
       }
       \left(
        X
       \right)
       =
       2
    \), which is given above, to the morphism
    \(
       f
    \),
    we can conclude that the semiorthogonal decomposition is indeed trivial.

    Finally, suppose that the general fiber of
    \(
       h
    \)
    is of genus
    \(
       1
    \).
    Then it follows from the canonical bundle formula for elliptic fibrations~(see, say,~\cite[Section~4~(4.1)]{2015arXiv150800682K}) that there exists an open dense subset
    \(
       V
       \hookrightarrow
       C
    \)
    such that
    
    \noindent
    \begin{align}
        f
        \left(
            \baselocus
            _{
                f
            }
            (
                \omega
                _{
                    f
                }
            )
        \right)
        \subseteq
        C
        \setminus
        V.
    \end{align}
    Hence one of the components of the semiorthogonal decomposition is supported in
    \(
       f
       ^{
        - 1
       }
       \left(
        C
        \setminus
        V
       \right)
    \)
    by~\cref{theorem:dichotomy}.
    Then the proof of~\cite[Theorem~4.2]{2015arXiv150800682K} literally applies to our situation. Thus we conclude the proof.
\end{proof}

To conclude, let us summarize the remaining cases of~\cref{conjecture:main conjecture} in dimension
\(
   2
\):

\begin{corollary}\label{corollary:remaining cases}
    \cref{conjecture:main conjecture} is true except possibly for those
    \(
       X
    \)
    which satisfy all of the following conditions.

    \noindent
    \begin{itemize}
        \item \( X \) is a minimal surface of general type.
        \item \( H ^{ 2 } ( X, \cO _{ X } ) \neq 0 \) (\( \iff H ^{ 0 } ( X, \omega _{ X } ) \neq 0 \)).
        \item \( H ^{ 1 } ( X, \cO _{ X } ) = 0 \).
        \item There exists a 1-dimensional component of
        \(
           \baselocus
           \left(
            \omega _{ X }
           \right)
        \)
        which can not be contracted to a point in the category of algebraic spaces.
    \end{itemize}
\end{corollary}

\begin{proof}
    The case where
    \(
       X
    \)
    is of Kodaira dimension
    \(
       0
    \)
    is established in~\cite[]{2015arXiv150800682K}. The case where
    \(
        X
    \)
    is of Kodaira dimension
    \(
        1
    \)
    is also established in~\cite[]{2015arXiv150800682K}, except the case where
    \(
       \dim
       H ^{ 1 }
       (
        X,
        \cO _{ X }
       )
       =
       1
    \)
    and
    \(
        \dim
        H ^{ 2 }
        (
         X,
         \cO _{ X }
        )
        =
        0
    \).
    This case is partially settled in~\cite[Theorem~4.14]{2021arXiv210709564L}, and fully by~\cref{theorem:main}.

    Hence it remains to consider the case when
    \(
       X
    \)
    is of general type with
    \(
        \dim
        H ^{ 1 }
        (
         X,
         \cO _{ X }
        )
        =
        0
    \).
    We can also assume
    \( H ^{ 2 } ( X, \cO _{ X } ) \neq 0 \),
    since otherwise line bundles on
    \(
       X
    \)
    are exceptional objects and hence we can conclude that
    \(
       X
    \)
    is not semiorthogonally indecomposable.
    If any 1-dimensional components of
    \(
       \baselocus
       (
        \omega
        _{
            X
        }
       )
    \)
    is contractible to a point in the category of algebraic spaces, then we can show the indecomposability (this is done in~\cite[]{2015arXiv150800682K}, but we can also prove it by~\cref{lemma:any SOD is f-linear}). Hence we obtain the assertion.
\end{proof}

%
%------------------------------------------------------------------------
%
\section{Alternative proof via relative Hochschild homology}\label{section:Alternative proof via relative Hochschild homology}

Here we give an alternative proof of~\cref{theorem:dichotomy} using the relative Hochschild homology, under the extra assumption that
\(
   f
\)
is smooth. Indeed, the result is finer than~\cref{theorem:dichotomy}. This is the relative version of the absolute case (i.e., the case
\(
   Y
   =
   \Spec \bfk
\))
which is worked out in~\cite[Section~5]{2020arXiv200607643P}.

In what follows we freely use the notation in the following diagram.

\noindent
\begin{equation}
    \begin{tikzcd}
        X
        \arrow[
            rrd,
            bend left,
            "\id _{ X }"
        ]
        \arrow[
            rd,
            "\Delta
            _{
                f
            }"
        ]
        \arrow[
            rdd,
            bend right,
            "\id _{ X }"'
        ]
        &
        &
        \\
        &
       X
       \times
       _{
        Y
       }
       X
       \arrow[
        r,
        "\pr _{ 2 }"
       ]
       \arrow[
        d,
        "\pr _{ 1 }"'
       ]
       \arrow[
        rd,
        "f ^{ 2 }"
       ]
       &
       X
       \arrow[
        d,
        "f"
       ]
       \\
       &
       X
       \arrow[
        r,
        "f"'
       ]
       &
       Y
    \end{tikzcd}
\end{equation}

In this section we discuss an
\(
   f
\)-linear semiorthogonal decomposition

\noindent
\begin{align}
    \derived
    (
        X
    )
    =
    \langle
        \cA,
        \cB
    \rangle.
\end{align}
Then we obtain the base change

\noindent
\begin{align}\label{equation:base change SOD}
    \perf
    (
        X
        \times
        _{
            Y
        }
        X
    )
    =
    \langle
        \pr _{ 1 }
        ^{
            \ast
        }
        \cA,
        \pr _{ 1 }
        ^{
            \ast
        }
        \cB
    \rangle
\end{align}
which is automatically
\(
   \pr _{ 2 }
\)-linear (see~\cite[Section~5.1]{MR2801403}). Let us define the objects
\(
   K _{ \cA },
   K _{ \cB }
   \in
   \perf
   (
    X
    \times
    _{
        Y
    }
    X    
   )
\)
by the following distinguished triangle which is obtained from the semiorthogonal decomposition~\eqref{equation:base change SOD}. Note that
\(
    K _{ \cA },
    K _{ \cB }   
\)
are the kernels of the projection functors to the subcategories
\(
   \cA
\)
and
\(
   \cB
\),
respectively (see~\cite[Theorem~3.7]{kuznetsov2009hochschild} for details).

\noindent
\begin{align}
    K
    _{
        \cB
    }
    \to
    \Delta
    _{
     f
     \ast
    }
    \cO
    _{
     X
    }
    \to    
    K
    _{
        \cA
    }
    \to
    K
    _{
        \cB
    }
    [
        1
    ]
\end{align}

\begin{definition}
    In the situation of~\cref{theorem:dichotomy}, the \emph{relative Hochschild homology} of
    \(
       X
    \)
    over
    \(
       Y
    \),
    \(
       \cA
    \),
    and
    \(
       \cB
    \)
    is defined as follows, respectively.
    
    \noindent
    \begin{align}
        \HH
        _{
            \ast
        }
        \left(
            X / Y
        \right)
        &
        \coloneqq
        f ^{ 2 }
        _{
            \ast
        }
        \cRHom
        _{
            X
            \times
            _{
                Y
            }
            X
        }
        \left(
            \Delta
            _{
                f
                \ast
            }
            \cO
            _{
                X
            },
            \Delta
            _{
                f
                \ast
            }
            \omega            
            _{
                X
            }
            [
                \dim X
            ]
        \right)
        \\
        \HH
        _{
            \ast
        }
        \left(
            \cA
        \right)
        &
        \coloneqq
        f ^{ 2 }
        _{
            \ast
        }
        \cRHom
        _{
            X
            \times
            _{
                Y
            }
            X
        }
        \left(
            K
            _{
                \cA
            },
            K
            _{
                \cA
            }
            \circ
            \Delta
            _{
                f
                \ast
            }
            \omega
            _{
                X
            }
            [
                \dim X
            ]
        \right)
        \\
        \HH
        _{
            \ast
        }
        \left(
            \cB
        \right)
        &
        \coloneqq
        f ^{ 2 }
        _{
            \ast
        }
        \cRHom
        _{
            X
            \times
            _{
                Y
            }
            X
        }
        \left(
            K
            _{
                \cB
            },
            K
            _{
                \cB
            }
            \circ
            \Delta
            _{
                f
                \ast
            }
            \omega
            _{
                X
            }
            [
                \dim X
            ]
        \right)
    \end{align}
\end{definition}

Then one can naturally define a morphism of
\(
    \bZ
\)-graded coherent sheaves on
\(
    Y
\)

\noindent
\begin{align}\label{equation:decomposition of Hochschild homology}
    \gamma
    _{
        \cA
    }
    \oplus
    \gamma
    _{
        \cB
    }
    \colon
    \HH
    _{
        \ast
    }
    \left(
        X / Y
    \right)
    \to
    \HH
    _{
        \ast
    }
    \left(
        \cA
    \right)
    \oplus
    \HH
    _{
        \ast
    }
    \left(
        \cB
    \right)
\end{align}
as follows.

Choose
\(
    m
    \in
    \bZ
\).
To define the morphism

\noindent
\begin{align}
    \gamma
    _{
        \cA,
        m
    }
    \colon
    \HH
    _{
        m
    }
    \left(
        X / Y
    \right)
    \to
    \HH
    _{
        m
    }
    \left(
        \cA
    \right)
\end{align}
by using the Yoneda lemma, take an arbitrary object
\(
    \cE
    \in
    \perf ( Y )
\).
All we have to do is to construct a map
    
\noindent
\begin{align}
    \Hom
    _{
        \perf ( Y )
    }
    \left(
        \cE,
        \HH
        _{
            m
        }
        \left(
            X / Y
        \right)
    \right)
    \to
    \Hom
    _{
        \perf ( Y )
    }
    \left(
        \cE,
        \HH
        _{
            m
        }
        \left(
            \cA
        \right)
    \right)
\end{align}
which is functorial in
\(
    \cE
\).
By the standard adjunctions it is equivalent to giving a map as follows.

\noindent
\begin{align}
    \Hom
    _{
        \perf
        (
            X
            \times
            _{
                Y
            }
            X
        )
    }
    \left(
        f
        ^{
            2
            \ast
        }
        \cE
        \otimes
        \Delta
        _{
            f
            \ast
        }
        \cO
        _{
            X
        },
        \Delta
        _{
            f
            \ast
        }
        \omega            
        _{
            X
        }
        [
            \dim X
            +
            m
        ]            
    \right)
    \\
    \to
    \Hom
    _{
        \perf
        (
            X
            \times
            _{
                Y
            }
            X
        )
    }
    \left(
        f
        ^{
            2
            \ast
        }
        \cE
        \otimes
        K
        _{
            \cA
        },
        K
        _{
            \cA
        }
        \circ
        \Delta
        _{
            \ast
        }
        \omega
        _{
            X
        }
        [
            \dim X
            +
            m
        ]
    \right)
\end{align}
This is achieved by the same arguments of~\cite[Section~5.1]{2020arXiv200607643P}, so we omit the details. The morphism
\(
   \gamma
   _{
    \cB,
    m
   }
\)
is also obtained from the same arguments.
Finally, by adjusting the proof of~\cite[Proposition~5.5]{kuznetsov2009hochschild} to the relative setting, we obtain the following

\begin{proposition}\label{proposition:Hochschild homology decomposes}
    The morphism~\eqref{equation:decomposition of Hochschild homology} is an isomorphism.
\end{proposition}

\begin{proof}
    See~\cite[Proposition~5.5,~Theorem~7.3,~Corollary~3.10]{kuznetsov2009hochschild} and their proofs for the absolute case. We can adjust them to the relative setting.
\end{proof}

\begin{remark}
    More abstract proof is given (in the absolute setting) in~\cite[Theorem~4.7]{2021arXiv211210312L}. We opted to follow the concrete approach of~\cite[]{kuznetsov2009hochschild} in this paper.
\end{remark}

As in the absolute case, we have the following assertion as to the decomposition~\eqref{equation:decomposition of Hochschild homology}.

\begin{lemma}
    Either of the following assertions holds.
    \begin{enumerate}[(A)]
        \item\label{item:HH(cA)=0}
        \(
            \HH
            _{
                - \dim X
            }
            \left(
                \cA
            \right)
            =
            0       
        \)

        \item\label{item:HH(cB)=0}
        \(
            \HH
            _{
                - \dim X
            }
            \left(
                \cB
            \right)
            =
            0
        \)
    \end{enumerate}
    Moreover, when the case~\cref{item:HH(cA)=0} holds, then for any
    \(
       a
       \in
       \cA
    \)
    the canonical map
    
    \noindent
    \begin{align}\label{equation:action of omega on a}
        f
        ^{
            \ast
        }
        f
        _{
            \ast
        }
        \omega
        _{
            X
        }
        \otimes
        _{
            \cO
            _{
                X
            }
        }
        a
        \to
        \omega
        _{
            X
        }
        \otimes
        _{
            \cO
            _{
                X
            }
        }
        a
    \end{align}
    is
    \(
       0
    \).
    Hence for any \'etale morphism from an affine scheme as in~\eqref{equation:etale chart} and a section
    \(
       s
       \in
       H ^{ 0 }
       \left(
        X ',
        \omega
        _{
            X '
        }
       \right)
    \)
    (see~\eqref{equation:cartesian diagram} for the definition of
    \(
       X '
    \))
    the multiplication map
    
    \noindent
    \begin{align}
        a '
        \xrightarrow[]{
            \otimes
            s
        }
        a '
        \otimes
        \omega
        _{
            X '
        }
    \end{align}
    is
    \(
       0
    \). In particular, it holds that
    \(
       \Supp
       (
        a
       )
       \subset
       \baselocus
       _{
        f
       }
       (
        \omega
        _{
            X
        }
       )
    \).    
    If the case~\cref{item:HH(cB)=0} holds, then the same assertions hold for objects of
    \(
       \cB
    \).
\end{lemma}

\begin{proof}
    Without loss of generality, we can assume that
    \(
       f
    \)
    is surjective and hence
    \(
       Y
    \)
    is also connected.
    % Hence it is enough to show that either~\cref{item:HH(cA)=0} or~\cref{item:HH(cB)=0} holds after we replace
    % \(
    %    X
    % \)
    % with
    % \(
    %    X '
    % \)
    % by an \'etale morphism from an affine scheme as in~\eqref{equation:etale chart}.
    % Therefore we may and will assume that
    % \(
    %    Y
    % \)
    % is an affine scheme.

    Suppose that
    \(
       \HH
       _{
        - \dim X
       }
       \left(
        \cB
       \right)
       \neq
       0
    \).
    Let
    \(
       X '
    \)
    be the base change of
    \(
       X
    \)
    as in~\eqref{equation:cartesian diagram} by an \'etale morphism from an affine scheme as in~\eqref{equation:etale chart} such that

    \noindent
    \begin{align}
        \HH
        _{
        - \dim X
        }
        \left(
        \cB '
        \right)
        \simeq
        y
        ^{
        \ast
        }
        \HH
        _{
        - \dim X
        }
        \left(
        \cB
        \right)
        \neq
        0
    \end{align}
    
    From the isomorphism~\eqref{equation:decomposition of Hochschild homology} in degree
    \(
       - \dim X
    \)
    we obtain a section
    \(
        0
        \neq
       s
       \in
       H ^{ 0 }
       \left(
        X ',
        \omega
        _{
            X '
        }
       \right)
    \)
    such that
    \(
       \gamma
       _{
        \cA '
       }
       (
        s
       )
       =
       0
    \).
    This implies that the objects of
    \(
       \cA '
    \)
    are supported in the vanishing locus of
    \(
       s
    \),
    which means that there is a proper closed subset
    \(
        Z
        _{
            \cA
        }
        \subsetneq
       X
    \)
    in which objects of
    \(
       \cA
    \)
    are supported.
    Likewise, if
    \(
       \HH
       _{
        - \dim X
       }
       \left(
        \cA
       \right)
       \neq
       0
    \),
    then it follows that there is a proper closed subset
    \(
        Z
        _{
            \cB
        }
        \subsetneq
       X
    \)
    in which objects of
    \(
       \cB
    \)
    are supported. Then the skyscraper sheaves (if any) of the closed points of the locus
    
    \noindent
    \begin{align}\label{equation:contradictory locus}
        X
        \setminus
        \left(
         Z
         _{
             \cA
         }
         \cup
         Z
         _{
             \cB
         }
        \right)
    \end{align}
    should be orthogonal to both
    \(
       \cA
    \)
    and
    \(
       \cB
    \), which is a contradiction. As we assume that
    \(
       X
    \)
    is connected, it follows that either
    \(
       Z
       _{
        \cA
       }
    \)
    or
    \(
       Z
       _{
        \cB
       }
    \)
    is empty.
    
    To see the rest of the assertions, note
        
    \noindent
    \begin{align}
        &
        f
        _{
            \ast
        }
        \omega
        _{
            X
        }
        \stackrel{
            \gamma
            _{
             \cA,
             - \dim X
            }            
        }{
            \to
        }
        f
        ^{
            2
        }
        _{
            \ast
        }
        \cHom
        _{
            X
            \times
            _{
                Y
            }
            X
        }
        \left(
            K
            _{
                \cA
            },
            K
            _{
                \cA
            }
            \circ
            \Delta
            _{
                \ast
            }
            \omega
            _{
                X
            }
        \right)
        \\
        \iff
        &
        f
        ^{
            2
            \ast
        }
        f
        _{
            \ast
        }
        \omega
        _{
            X
        }
        \otimes
        K
        _{
            \cA
        }
        \to
        K
        _{
            \cA
        }
        \circ
        \Delta
        _{
            \ast
        }
        \omega
        _{
            X
        }\label{equation:map of kernels}
    \end{align}
    Think of~\eqref{equation:map of kernels} as a morphism of kernels.
    Take
    \(
       a
       \in
       \cA
    \)
    and evaluate at
    \(
        \omega
        _{
            X
        }
        ^{
            - 1
        }            
        \otimes
        a
    \),
    to obtain the first row of the following commutative diagram
    (see the corresponding diagram in~\cite[Proof of Lemma~5.4]{2020arXiv200607643P}), where
    \(
       \Phi
       _{
        \bullet
       }
    \)
    is our notation for the integral transformation by the kernel
    \(
       \bullet
    \).
    The left vertical arrow is the adjoint unit
    
    \noindent
    \begin{align}
        \id
        _{
            \derived
            (
                X
            )
        }
        \to
        \Phi
        _{
            K
            _{
                \cA
            }
        }
    \end{align}
    tensored with the identity of
    \(
        f
        ^{
            2
            \ast
        }
        f
        _{
            \ast
        }
        \omega
        _{
            X
        }       
    \).
    
    \noindent
    \begin{equation}
        \begin{tikzcd}
            f
            ^{
                2
                \ast
            }
            f
            _{
                \ast
            }
            \omega
            _{
                X
            }
            \otimes
            \omega
            _{
                X
            }
            ^{
                - 1
            }            
            \otimes
            a
            \arrow[
                rr,
                "\eqref{equation:action of omega on a}"
            ]
            \arrow[d]
            &
            &
            a
            \arrow[
                d,
                "\id _{ a }"
            ]
            \\
            f
            ^{
                2
                \ast
            }
            f
            _{
                \ast
            }
            \omega
            _{
                X
            }
            \otimes
            \Phi
            _{
                K
                _{
                    \cA
                }
            }
            \left(
                \omega
                _{
                    X
                }
                ^{
                    - 1
                }            
                \otimes
                a
            \right)
            \arrow[
                rr,
                "\eqref{equation:map of kernels}
                \left(
                    \omega
                    _{
                        X
                    }
                    ^{
                        - 1
                    }            
                    \otimes
                    a
                \right)"'
            ]
            &
            &
            a
        \end{tikzcd}
    \end{equation}   

    The bottom arrow is trivial as we assumed
    \(
       \gamma
       _{
        \cA
       }
       =
       0
    \).
    Thus we see
    \(
       {\eqref{equation:action of omega on a}}
       =
       0
    \).
    The rest is straightforward.
\end{proof}

\printbibliography%
\end{document}